\documentclass[12pt, reqno]{article}
\usepackage{amsmath}
\usepackage{amssymb}
\usepackage{latexsym,bm}
\usepackage{amsthm}
\usepackage{graphicx}
\usepackage{cite}
\usepackage[colorlinks,urlcolor=blue]{hyperref}

\usepackage[top=3cm,bottom=3cm,left=3cm,right=3cm,includehead,includefoot]{geometry}
\usepackage{algorithm}
\usepackage{algorithmic}
\allowdisplaybreaks

\newtheorem{thm}{Theorem}[section]

\newtheorem{lem}{Lemma}[section]

\theoremstyle{definition}
\newtheorem{defn}{Definition}[section]
\theoremstyle{Condition}
\newtheorem{Cond}{Condition}[section]
\theoremstyle{remark}
\newtheorem{rmk}{Remark}[section]
\numberwithin{equation}{section}
\theoremstyle{example}

\numberwithin{equation}{section}

\begin{document}

\bigskip
\bigskip

\bigskip

\begin{center}

\textbf{\large Efficient primal-dual fixed point algorithm with
dynamic stepsize for convex problems with applications to imaging restoration}

\end{center}

\begin{center}
Meng Wen $^{1,2}$, Shigang Yue$^{4}$, Yuchao Tang$^{3}$, Jigen Peng$^{1,2}$
\end{center}

\begin{center}
1. School of Mathematics and Statistics, Xi'an Jiaotong University,
Xi'an 710049, P.R. China \\
2. Beijing Center for Mathematics and Information Interdisciplinary
Sciences, Beijing, P.R. China

 3. Department of Mathematics, NanChang University, Nanchang
330031, P.R. China\\
4. School of Computer Science, University of Lincoln, LN6 7TS, UK
\end{center}

\footnotetext{\hspace{-6mm}$^*$ Corresponding author.\\
E-mail address: wen5495688@163.com}

\bigskip

\noindent  \textbf{Abstract} In this paper we consider the problem
of finding the minimization of the sum  of a convex function and
the composition of another convex function with a continuous linear
operator from the view of fixed point algorithms based on proximity
operators. We design a primal-dual fixed point algorithm with
dynamic stepsize based on the proximity
operator(PDFP$^{2}O_{DS_{n}}$ for $a_{n}\subset(0,1)$)and obtain a
scheme with a closed-form solution for each iteration. Based on
Modified Mann iteration and the firmly nonexpansive properties of
the proximity operator, we achieve the convergence of the proposed
PDFP$^{2}O_{DS_{n}}$ algorithm. Moreover, under some stronger
assumptions, we can prove the global linear convergence of the
proposed algorithm. We also give the connection of the proposed
algorithm with other existing first-order methods and fixed point
algorithms FP$^{2}O$(Micchelli et al
2011 Inverse Problems 27 45009-38), PDFP$^{2}O$(Chen et al
2013 Inverse Problems 29). Finally, we illustrate the
efficiency of PDFP$^{2}O_{DS_{n}}$ through some numerical examples
on  the CT image reconstruction problem. Generally speaking, our
method PDFP$^{2}O_{DS}$ is comparable with other state-of-the-art
methods in numerical performance, while it has some advantages on
parameter selection in real applications and converges faster than
PDFP$^{2}O$.


\noindent \textbf{MR(2000) Subject Classification} 47H09, 90C25,

\section{Introduction}

The purpose of this paper is to designing and discussing an
efficient algorithmic framework  with dynamic stepsize for
minimizing the sum  of a convex function and the composition of
another convex function with a continuous linear operator, i.e.
$$\min (f_{1}\circ D)(x)+f_{2}(x),\eqno{(1.1)}$$
where $f_{1}\in\Gamma_{0}(\mathbb{R}^{m})$,
$f_{2}\in\Gamma_{0}(\mathbb{R}^{n}),$ and $f_{2}$ is differentiable
on $\mathbb{R}^{n}$ with a $1/\beta$-Lipschitz continuous gradient
for some $\beta\in(0,+\infty)$ and
$D:\mathbb{R}^{n}\rightarrow\mathbb{R}^{m}$ a linear transform. This
parameter $\beta$ is related to the convergence conditions of
algorithms presented in the following section. Here and in what
follows, for a real Hilbert space $\mathcal{H}$,
$\Gamma_{0}(\mathcal{H})$ denotes the collection of all proper lower
semi-continuous convex functions from $\mathcal{H}$ to
$(-\infty,+\infty]$. Despite its simplicity, many problems in image
processing can be  translated into the form of (1.1). For example, the
following variational sparse recovery models are often considered in
image restoration and medical image reconstruction:
$$\min\frac{1}{2}\|Ax-b\|_{2}^{2}+\lambda\psi(Dx),\eqno{(1.2)}$$
where $\|\cdot\|_{2}$ denotes the usual Euclidean norm for a vector,
$A\in\mathbb{R}^{p\times n}$ describes  a blur operator, $b\in
\mathbb{R}^{p}$ represents the blurred and noisy image  and
$\lambda>0$ is the regularization parameter in the context of
deblurring and denoising of images. The class of regularizers (1.2)
includes a plethora of methods, depending on the choice of the
function $\psi$ and of matrix $D$. Our motivation for studying this
class of penalty functions arises from sparsity inducing
regularization methods which consider $\psi$ to be either the
$l_{1}$ norm or a mixed $l_{1}-l_{2}$ norm. When $D$ is the identity
matrix, the latter case corresponds to the well-known
Group Lasso method [15], for which well studied optimization
techniques are available. Other choices of the matrix $D$ give rise
to different kinds of Group Lasso with overlapping groups [16-17],
which have proved to be effective in modeling structured sparse
regression problems. Problem (1.2) can be expressed in the form of
(1.1) by setting $f_{1}=\lambda\psi$,
$f_{2}=\frac{1}{2}\|Ax-b\|_{2}^{2}$. One of the main difficulties in
solving it is that  $\psi$ are non-differentiable. The case often
occurs in many problems we are interested in.
\par
  For
problem (1.1), Peijun Chen, Jianguo Huang and Xiaoqun Zhang proposed
 a primal-dual fixed point algorithm($PDFP^{2}O)$ in [1], i.e.
$$
\left\{
\begin{array}{l}
v_{n+1}=(I-prox_{\frac{\gamma}{\lambda}f_{1}})(D(x_{n}-\gamma\nabla f_{2}(x_{n}))+(I-\lambda DD^{T})v_{n}),\\
x_{n+1}=x_{n}-\gamma\nabla f_{2}(x_{n})-\lambda D^{T}v_{n+1},
\end{array}
\right.\eqno{(1.3)}
$$
where $0<\lambda\leq1/\lambda_{\max}(DD^{T})$, $0<\gamma<2\beta$,
and the operator $ prox_{f}$ is defined by
\begin{align*}
prox_{f}&:\mathcal{H}\rightarrow\mathcal{H}\\
& x\mapsto \arg \min_{y\in H}f(y)+\frac{1}{2}\|x-y\|_{2}^{2},
\end{align*}
called the proximity operator of $f$ . Note that this type of
splitting method was originally studied in [1,8] and the notion of
proximity operators was first introduced by Moreau in [9] as a
generalization of projection operators. For general $D$ and $f_{2}$,
each step of the proposed algorithm is explicit when
$prox_{\frac{\gamma}{\lambda}f_{1}}$ is easy to compute. However,
the proximity operators for the general form $f = f_{1} \circ D$ as
in (1.1) do not have an explicit expression, leading to the
numerical solution of a difficult subproblem. In fact for
$\lambda\psi=\mu\|\cdot\|$, the subproblem of (1.2) is
$$\min\frac{1}{2}\|x-b\|_{2}^{2}+\mu\|Dx\|,\eqno{(1.4)}$$
where $A\in\mathbb{R}^{p\times n}$ describes  a blur operator, $b
\in\mathbb{R}^{p}$ denotes a corrupted image to be denoised.
\par
 The obvious advantage of the algorithm($PDFP^{2}O)$ proposed by Chen et al [1] for problem (1.1) is
that it is very easy for parallel implementation. However,
in this paper we aim to provide a more general iteration in which
the coefficient $\gamma$ is made iteration-dependent to solve the
general problem (1.1), errors are allowed in the evaluation of the
operators $prox_{\frac{\gamma}{\lambda}f_{1}}$ and $\nabla f_{2}$,
and a relaxation sequence $\lambda_{n}$ is introduced. The errors
allow for some tolerance in the numerical implementation of the
algorithm, while the flexibility introduced by the
iteration-dependent parameters  $\gamma_{n}$ and $\lambda_{n}$ can
be used to improve its convergence pattern. In addition, we will
reformulate our fixed point type of methods and show their
connections with some existing first-order methods and primal-dual
fixed point algorithm for (1.1) and (1.2).

The rest of this paper is organized as follows. In the next section,
 we recall the  primal-dual fixed point
algorithm($PDFP^{2}O)$ and some related works and then deduce the
proposed PDFP$^{2}O_{DS}$ algorithm and its extension
PDFP$^{2}O_{DS_{n}}$ from our intuitions. In section 3, we first
deduce PDFP$^{2}O_{DS_{n}}$ again in the setting of fixed point
iteration; we then establish its convergence under a general setting
and the convergence rate under some stronger assumptions on $\nabla
f_{2}$ and $D$. In section 4, we give the equivalent form of
PDFP$^{2}O_{DS}$, and the relationships and differences with other
first-order algorithms. In the final section, we show the numerical
performance and efficiency of PDFP$^{2}O_{DS_{n}}$ through some
examples on on  the CT image reconstruction problem  and compare
their performances to the ones of some iterative schemes recently
introduced in the literature.

\section{Fixed Point Algorithms Based on Proximity Operators}

Similar to the proximity algorithms(FP$^{2}$O) for Image Models:
Denoising proposed by Micchelli et al [8], Andreas Argyriou et al
proposed an algorithm called IFP$^{2}$O in [10] to solve
$$\min (f_{1}\circ D)(x)+\frac{1}{2}x^{T}Qx-b^{T}x ,$$
where $x\in\mathbb{R}^{n}$, $Q\in M_{n}$, with $M_{n}$ being the
collection of all symmetric positive definite $n\times n$ matrices,
$b\in\mathbb{R}^{n}$. Define
$$H(v)=(I-prox_{\frac{f_{1}}{\lambda}})(DQ^{-1}b+(I-\lambda DQ^{-1}D^{T})v) \,\,for \,all\,\, v\in \mathbb{R}^{m}.$$
Then, the corresponding algorithm is given below, called algorithm
1, which can be viewed as a fixed point algorithm based on the
inverse matrix and proximity operator(IF
P$^{2}$O). Here $H_{\kappa}$
is the $\kappa$-averaged operator of $H$, i.e. $H_{\kappa} = \kappa
I + (1 - \kappa)H$ for $\kappa \in (0, 1)$; see definition 3.3 in
the following section, the matrix  $Q$ is assumed to be invertible
and the inverse can be easily calculated, which is unfortunately not
the case in most of the applications in imaging science. Moreover,
there is no theoretical guarantee of convergence if the linear
system is only solved approximately.

\begin{algorithm}[H]
\caption{FP$^{2}$O based on inverse matrix, IFP$^{2}$O [10].}
\begin{algorithmic}\label{1}
\STATE Step 1: Choose $v_{0}\in \mathbb{R}^{m}$,
$0<\lambda\leq2/\lambda_{\max}(DQ^{-1}D^{T})$, $\kappa\in(0,1)$.\\
Step 2: calculate $v^{\ast} $, which is the fixed point of $H$, with
iteration $v_{n+1}=H_{\kappa}(v_{n})$.\\
Step 3: $x^{\ast}=Q^{-1}(b-\lambda D^{T}v^{\ast} ). $
\end{algorithmic}
\end{algorithm}
 The authors in [10] combined a proximal forward-backward
splitting (PFBS) algorithm proposed by Combettes and Wajs [2] and
FP$^{2}$O for solving problem (1.3), for which we call
PFBS$_{-}$FP$^{2}$O (cf algorithm 2 below). Precisely speaking, at
step k in PFBS, after one forward iteration
$x_{n+1/2}=x_{n}-\gamma\nabla f_{2}(x_{n}) $, we need to solve for
$x_{n+1}=prox_{\gamma f_{1}\circ D}(x_{n+1/2})$. FP$^{2}$O is then
used to solve this subproblem, i.e. the fixed point $v_{n+1}^{\ast}$
of $H_{x_{n+1/2}}$ is obtained by the fixed iteration form
$\underline{v}_{k+1}=(H_{x_{n+1/2}})_{\kappa}(\underline{v}_{k})$,
where
$$H_{x_{n+1/2}}(v)=(I-prox_{\gamma f_{1}\circ D})(Dx_{n+1/2}+(I-\lambda DD^{T})v) \,\,for \,all\,\, v\in \mathbb{R}^{m}.\eqno{(2.1)}$$
Then $x_{n+1}$ is given by setting $x_{n+1}=x_{n+1/2}-\lambda
D^{T}v_{n+1}^{\ast}$. The acceleration combining with the Nesterov
method [11-14] was also considered in [10]. But the algorithm 2
involves inner and outer iterations, and it is often problematic to
set the appropriate inner stopping conditions to balance
computational time and precision.
\begin{algorithm}[H]
\caption{Proximal forward-backward splitting based on FP$^{2}$O,
PFBS$_{-}$FP$^{2}$O [10].}
\begin{algorithmic}\label{1}
\STATE Step 1: Choose $x_{0}\in \mathbb{R}^{n}$, $0<\gamma<2\beta$.\\
Step 2: for $k=0,1,2,\ldots$\\
        $x_{n+1/2}=x_{n}-\gamma\nabla
        f_{2}(x_{n})$,\\
        calculate the fixed
point $v_{n+1}^{\ast}$ of $H_{x_{n+1/2}}$, with
iteration $\underline{v}_{n+1}=(H_{x_{n+1/2}})_{\kappa}(\underline{v}_{k})$,\\
$x_{n+1}=x_{n+1/2}-\lambda D^{T}v_{n+1}^{\ast}$.\\
end for
\end{algorithmic}
\end{algorithm}
Further, the authors in [1] suppose $\kappa = 0$ in FP$^{2}$O, the
idea is to take the numerical solution  $v_{n}$ of the fixed point
of $H_{x_{(n-1)+1/2}}$ as the initial value, and only perform one
iteration for solving the fixed point of $H_{x_{n+1/2}}$ ; then they
 obtained the  iteration scheme (1.4),  i.e.

$$
\left\{
\begin{array}{l}
v_{n+1}=(I-prox_{\frac{\gamma}{\lambda}f_{1}})(D(x_{n}-\gamma\nabla f_{2}(x_{n}))+(I-\lambda DD^{T})v_{n}),\\
x_{n+1}=x_{n}-\gamma\nabla f_{2}(x_{n})-\lambda D^{T}v_{n+1}.
\end{array}
\right.
$$
Then, the corresponding algorithm is given below, called algorithm
3. Since $v$ is actually the dual variable of the primal-dual form
related to (1.1), so algorithm 3 can be viewed as a primal-dual
fixed point algorithm based on the proximity operator(PDFP$^{2}O$).
\begin{algorithm}[H]
\caption{Primal-dual fixed point algorithm based on proximity
operator, PDFP$^{2}O$ [1].}
\begin{algorithmic}\label{1}
\STATE Initialization: Choose $x_{0}\in \mathbb{R}^{n}$, $v_{0}\in
\mathbb{R}^{m}$, $0<\lambda\leq1/\lambda_{\max}(DD^{T})$, $0<\gamma<2\beta$.\\
Iterations ($n\geq0$): Update $x_{n}$, $v_{n}$, $x_{n+\frac{1}{2}}$
as follows
$$
\left\{
\begin{array}{l}
x_{n+\frac{1}{2}}=x_{n}-\gamma_\nabla f_{2}(x_{n}),\\
v_{n+1}=(I-prox_{\frac{\gamma}{\lambda}f_{1}})(Dx_{n+\frac{1}{2}}+(I-\lambda DD^{T})v_{n}),\\
x_{n+1}=x_{n+\frac{1}{2}}-\lambda D^{T}v_{n+1}.
\end{array}
\right.
$$
\end{algorithmic}
\end{algorithm}
Moreover, borrowing the fixed point formulation of PDFP$^{2}O$, the
authors in [1] introduce a relaxation parameter $\kappa \in [0, 1)$
to obtain algorithm 4, which is exactly a Picard method with
parameters. If $\kappa = 0$, then PDFP$^{2}O_{\kappa}$ reduces to
PDFP$^{2}O$.
\begin{algorithm}[H]
\caption{PDFP$^{2}O_{\kappa}$ [1].}
\begin{algorithmic}\label{1}
\STATE Initialization: Choose $x_{0}\in \mathbb{R}^{n}$, $v_{0}\in
\mathbb{R}^{m}$, $0<\lambda\leq1/\lambda_{\max}(DD^{T})$, $0<\gamma<2\beta$, $\kappa\in[0,1)$.\\
Iterations ($n\geq0$): Update $x_{n}$, $v_{n}$, $x_{n+\frac{1}{2}}$
as follows
$$
\left\{
\begin{array}{l}
x_{n+\frac{1}{2}}=x_{n}-\gamma_\nabla f_{2}(x_{n}),\\
\tilde{v}_{n+1}=(I-prox_{\frac{\gamma}{\lambda}f_{1}})(Dx_{n+\frac{1}{2}}+(I-\lambda DD^{T})v_{n}),\\
\tilde{x}_{n+1}=x_{n+\frac{1}{2}}-\lambda
D^{T}\tilde{v}_{n+1},\\
v_{n+1}=\kappa v_{n}+(1-\kappa)\tilde{v}_{n+1},\\
x_{n+1}=\kappa x_{n}+(1-\kappa)\tilde{x}_{n+1}.
\end{array}
\right.
$$
\end{algorithmic}
\end{algorithm}
The fixed point characterization provided by Peijun Chen et al [1]
suggests solving Problem (1.1 ) via the fixed point iteration scheme
(1.3) for a suitable value of the parameter $\gamma$, $\lambda$.
This iteration, which is referred to as a primal-dual fixed point
algorithm for convex separable minimization with applications to
image restoration. A very natural idea is to provide a more general
iteration in which the coefficient $\gamma$ is made
iteration-dependent to solve the general problem (1.1), then  we can
obtain the following iteration scheme:
$$
\left\{
\begin{array}{l}
v_{n+1}=(I-prox_{\frac{\gamma_{n}}{\lambda_{n}}f_{1}})(D(x_{n}-\gamma_{n} \nabla f_{2}(x_{n}))+(I-\lambda_{n} DD^{T})v_{n}),\\
x_{n+1} =x_{n}-\gamma_{n} \nabla f_{2}(x_{n})-\lambda
_{n}D^{T}v_{n+1},
\end{array}
\right.\eqno{(2.2)}
$$
which produces our proposed method algorithm 5, described below.
This algorithm can also be deduced from the fixed point formulation,
whose detail we will give in the following section. On the other
hand, since the parameter $\gamma_{n}$ and $\lambda_{n}$ are
dynamic, so we call our method a primal-dual fixed point algorithm
based on proximity operator with dynamic stepsize, and abbreviate it
as PDFP$^{2}O_{DS}$. If $\gamma_{n}\equiv\gamma$,
$\lambda_{n}\equiv\lambda$ then  form (2.2) is equivalent to form
(1.3). So PDFP$^{2}O$ can be seen as a special case of
PDFP$^{2}O_{DS}$. Moreover, PFEP and FP$^{2}O$ are also the special
case of PDFP$^{2}O_{DS}$. We will show the connection to this
algorithm and other ones in section 4.

\begin{algorithm}[H]
\caption{ Primal-dual fixed point algorithm based on proximity
operator with dynamic stepsize PDFP$^{2}O_{DS}$}
\begin{algorithmic}\label{1}
\STATE Initialization: Choose $x_{0}\in \mathbb{R}^{n}$,  $v_{0}\in
\mathbb{R}^{m}$,
$0<\liminf_{n\rightarrow\infty}\gamma_{n}\leq\limsup_{n\rightarrow\infty}\gamma_{n}<2\beta$,
$0<\liminf_{n\rightarrow\infty}\lambda_{n}\leq\limsup_{n\rightarrow\infty}\lambda_{n}\leq1/\lambda_{\max}(DD^{T})$.\\
Iterations ($n\geq0$): Update $x_{n}$, $v_{n}$, $y_{n}$  as follows
$$
\left\{
\begin{array}{l}
z_{n+1}=x_{n}-\gamma_{n}\nabla f_{2}(x_{n}),\\
v_{n+1}=(I-prox_{\frac{\gamma_{n}}{\lambda_{n}}f_{1}})(Dz_{n+1}+(I-\lambda_{n}DD^{T})v_{n}),\\
x_{n+1}=z_{n+1}-\lambda_{n}D^{T}v_{n+1}.
\end{array}
\right.
$$

\end{algorithmic}
\end{algorithm}
Borrowing the fixed point formulation of PDFP$^{2}O_{DS}$, we can
introduce a relaxation parameter $\alpha_{n}\subset(0,1)$ to obtain
algorithm 6, which is exactly a Mann method with parameters. The
rule for parameter selection will be illustrated in section 3.
 Our theoretical analysis
for PDFP$^{2}O_{DS_{n}}$ given in the following section is mainly
based on this fixed point setting.

\begin{algorithm}[H]
\caption{  PDFP$^{2}O_{DS_{n}}$}
\begin{algorithmic}\label{2}
\STATE Initialization: Choose $x_{0}\in \mathbb{R}^{n}$,  $v_{0}\in
\mathbb{R}^{m}$,
$0<\liminf_{n\rightarrow\infty}\gamma_{n}\leq\limsup_{n\rightarrow\infty}\gamma_{n}<2\beta$,
$0<\liminf_{n\rightarrow\infty}\lambda_{n}\leq\limsup_{n\rightarrow\infty}\lambda_{n}\leq1/\lambda_{\max}(DD^{T})$, $\alpha_{n}\subset(0,1)$.\\
Iterations ($n\geq0$): Update $x_{n}$, $v_{n}$, $y_{n}$  as follows
$$
\left\{
\begin{array}{l}
z_{n+1}=x_{n}-\gamma_{n}\nabla f_{2}(x_{n}),\\
\tilde{v}_{n+1}=(I-prox_{\frac{\gamma_{n}}{\lambda_{n}}f_{1}})(Dz_{n+1}++(I-\lambda_{n}DD^{T})v_{n})\\
\tilde{x}_{n+1}=z_{n+1}-\lambda_{n}D^{T}\tilde{v}_{n+1},\\
v_{n+1}=\alpha_{n}v_{n}+(1-\alpha_{n})\tilde{v}_{n+1},\\
x_{n+1}=\alpha_{n}x_{n}+(1-\alpha_{n})\tilde{x}_{n+1}.
\end{array}
\right.
$$

\end{algorithmic}
\end{algorithm}
\section{Convergence analysis}
\subsection{General convergence}
First of all, let us mention some related definitions and lemmas for
later requirements. We always assume
that problem (1.1) has at least one solution. As shown in [2], if
the objective function $(f_{1}\circ D)(x)+f_{2}(x)$ is coercive,
i.e.
$$\lim_{\|x\|^{2}\rightarrow+\infty}((f_{1}\circ D)(x)+f_{2}(x))=+\infty,$$
then the existence of solution can be ensured for (1.1).
\begin{defn}
(Subdifferential [3]). Let $f$ be a function in
$\Gamma_{0}(\mathcal{H})$.  The subdifferential of $f$ is the
set-valued operator $\partial f : \mathcal{H} \rightarrow
2^{\mathcal{H}}$ , the value of which at $x \in  \mathcal{H}$ is
$$\partial f(x)=\{v\in\mathcal{H}|\langle v,y-x\rangle+f(x)\leq f(y)\,\,\, for\,\,all\,\,\, y\in \mathcal{H}^{2}\},$$
where $\langle\cdot,\cdot\rangle$ denotes the inner-product over
$\mathcal{H}$.
\end{defn}

\begin{defn}
(Nonexpansive operators and firmly nonexpansive operators [3]). An
operator $T : \mathcal{H} \rightarrow {\mathcal{H}}$ is nonexpansive
if and only if it satisfies
$$\|Tx-Ty\|_{2}\leq\|x-y\|_{2}\,\,\, for\,\,all\,\,\, (x,y)\in \mathcal{H}^{2}.$$
$T$ is firmly nonexpansive if and only if it satisfies one of the
following equivalent conditions:
\par
(i)$\|Tx-Ty\|_{2}^{2}\leq\langle Tx-Ty,x-y\rangle$\,\,\,
for\,\,all\,\,\, $(x,y)\in \mathcal{H}^{2}$.
\par
(ii)$\|Tx-Ty\|_{2}^{2}=\|x-y\|_{2}^{2}-\|(I-T)x-(I-T)y\|_{2}^{2}$\,\,\,
for\,\,all\,\,\, $(x,y)\in \mathcal{H}^{2}$.
\par
It is easy to show from the above definitions that a firmly
nonexpansive operator T is nonexpansive.
\end{defn}
\begin{lem}
Suppose $f\in\Gamma_{0}(\mathbb{R}^{m})$ and $x\in \mathbb{R}^{m}$.
Then there holds
$$y\in\partial f(x)\Longleftrightarrow x=prox_{f}(x+y).\eqno{(3.1)}$$
Furthermore, if $f$ has $1/\beta$-Lipschitz continuous gradient,
then
$$\langle\nabla f(x)-\nabla f(y),x-y\rangle\geq\beta\|\nabla f(x)-\nabla f(y)\|^{2}\,\,\, for\,\,all\,\,\, (x,y)\in\mathbb{R}^{m}.\eqno{(3.2)}$$
\end{lem}
\begin{proof}
The first result is nothing but proposition 2.6 of [4]. If $f$ has
$1/\beta$-Lipschitz continuous gradient, we have from [2] that
$\beta\nabla f$ is firmly nonexpansive, which implies (3.2) readily.

\end{proof}

\begin{lem}
(Lemma 2.4 of [2]). Let $f$ be a function in
$\Gamma_{0}(\mathbb{R}^{m})$. Then $prox_{f}$ and $I- prox_{f}$ are
both firmly nonexpansive operators.
\end{lem}
\begin{lem}
(The Resolvent Identity [5,6]). For $\lambda>0$ and $\nu>0$ and $x\in E$,\\
$$J_{\lambda}x=J_{\nu}(\frac{\nu}{\lambda}+(1-\frac{\nu}{\lambda})J_{\lambda}x).$$
\end{lem}
\begin{lem}
( [7]). Let $H$ be a real Hilbert space with inner product $\langle
\cdot,  \cdot \rangle$ and norm $\|\cdot\|$ , then \\
$\forall x,y\in H, \forall \alpha\in[0,1], \|\alpha
x+(1-\alpha)y\|^{2}=\alpha\|x\|^{2}+(1-\alpha)\|y\|^{2}-\alpha(1-\alpha)\|x-y\|^{2}.$

\end{lem}
\begin{lem}
([7]).  Let $C$ be a nonempty closed convex subset of   $H$,  $T :
C\rightarrow C$  is a nonexpansive mapping, and
$Fix(T)\neq\emptyset$. Then the mapping $I-T$ is demiclosed at zero,
that is $x_{n}\rightharpoonup x$ and $\|x_{n}-Tx_{n}\|\rightarrow
0$, then $x=Tx$.
\end{lem}
\par
The following lemmas are obtained from the reference [ 1 ].
\par
From reference [ 1 ], we know that for any two positive numbers
$\lambda$ and $\gamma$ , define
$T_{1}:\mathbb{R}^{m}\times\mathbb{R}^{n}\rightarrow\mathbb{R}^{m}$
as
$$T_{1}(v,x)=(I-prox_{\frac{\gamma}{\lambda}f_{1}})(D(x-\gamma \nabla f_{2}(x))+(I-\lambda DD^{T})v)\eqno{(3.3)}$$
and
$T_{2}:\mathbb{R}^{m}\times\mathbb{R}^{n}\rightarrow\mathbb{R}^{m}$
as
$$T_{2}(v,x)=x-\gamma\nabla f_{2}(x)-\lambda D^{T}\circ T_{1}.\eqno{(3.4)}$$

Denote
$$T(v,x)=(T_{1}(v,x),T_{2}(v,x)).\eqno{(3.5)}$$
\begin{lem}
Let $\lambda$ and $\gamma$  be two positive numbers. Suppose that
$\hat{x}$ is a solution of (1.1). Then there exists $\hat{v}\in
\mathbb{R}^{m}$ such that
$$
 \left\{
\begin{array}{l}
\hat{v}=T_{1}(\hat{v},\hat{x}),\\
\hat{x}=T_{2}(\hat{v},\hat{x}).
\end{array}
\right.
$$
In other words, $\hat{u }= (\hat{v},\hat{x})$ is a fixed point of
$T$. Conversely, if $\hat{u}\in\mathbb{R}^{m}\times\mathbb{R}^{n}$
is a fixed point of $T$, with $\hat{u }= (\hat{v},\hat{x})$,
$\hat{v}\in \mathbb{R}^{m}$, $\hat{x}\in \mathbb{R}^{n}$ then
$\hat{x}$ is a solution of (1.1).
\end{lem}

Denote
$$g(x)=x-\gamma \nabla f_{2}(x),\,\, for\,\,all\,\,\, x\in \mathbb{R}^{n}.\eqno{(3.6)}$$

$$M=I-\lambda DD^{T}.\eqno{(3.7)}$$
When $0<\lambda\leq1/\lambda_{max}(DD^{T})$, $M$ is a symmetric
positive semi-definite matrix, so we can define the semi-norm
$$\|V\|_{M}=\sqrt{\langle v,Mv\rangle},\,\, for\,\,all\,\,\, v\in \mathbb{R}^{m}.\eqno{(3.8)}$$
For an element $u=(v,x)\in \mathbb{R}^{m}\times\mathbb{R}^{n}$, with
$v\in \mathbb{R}^{m}$ and $x\in\mathbb{R}^{n}$, let
$$\|u\|_{\lambda}=\sqrt{\|x\|_{2}^{2}+\lambda\|v\|_{2}^{2}}.\eqno{(3.9)}$$
 We can easily see that
$\|\cdot\|_{\lambda}$ is a norm over the produce space
$\mathbb{R}^{m}\times\mathbb{R}^{n}$ whenever $\lambda > 0$.

 According to the definitions in
(3.3)-(3.5), the component form of $u_{n+1} = T (u_{n} )$ can be
expressed as
$$
\left\{
\begin{array}{l}
v_{n+1}=T_{1}(v_{n},x_{n})=(I-prox_{\frac{\gamma}{\lambda}f_{1}})(D(x_{n}-\gamma \nabla f_{2}(x_{n}))+(I-\lambda DD^{T})v_{n}),\\
x_{n+1}=T_{2}(v_{n},x_{n})=x_{n}-\gamma \nabla
f_{2}(x_{n})-\lambda D^{T}\circ T_{1}(v_{n},x_{n})\\
=x_{n}-\gamma \nabla f_{2}(x_{n})-\lambda D^{T}v_{n+1}.
\end{array}
\right.
$$
Therefore, the iteration $u_{n+1} = T (u_{n} )$ is equivalent to
(1.3).

\begin{lem}
If $0<\gamma<2\beta$, $0<\lambda\leq1/\lambda_{max}(DD^{T})$, then
$T$ is nonexpansive under the norm $\|\cdot\|_{\lambda}$.
\end{lem}

\begin{lem}
Suppose $0<\gamma<2\beta$, $0<\lambda\leq1/\lambda_{max}(DD^{T})$.
Let $u_{n} = (v_{n}, x_{n} )$ be the sequence generated by
$PDFP^{2}O$. Then the sequence $\{u_{n}\}$ converges to a fixed
point of $T$, and the sequence$\{x_{n}\}$ converges to a solution of
problem (1.1).

\end{lem}

Now, we are ready to discuss the convergence of
PDFP$^{2}O_{DS_{n}}$. To this end, let
$0<\liminf_{n\rightarrow\infty}\gamma_{n}\leq\limsup_{n\rightarrow\infty}\gamma_{n}<2\beta$,
$0<\liminf_{n\rightarrow\infty}\lambda_{n}\leq\limsup_{n\rightarrow\infty}\lambda_{n}\leq1/\lambda_{\max}(DD^{T})$
, define
$T_{1}^{n}:\mathbb{R}^{m}\times\mathbb{R}^{n}\rightarrow\mathbb{R}^{m}$
as
$$T_{1}^{n}(v,x)=(I-prox_{\frac{\gamma_{n}}{\lambda_{n}}f_{1}})(D(x-\gamma_{n} \nabla f_{2}(x))+(I-\lambda_{n} DD^{T})v)\eqno{(3.10)}$$
and
$T^{n}_{2}:\mathbb{R}^{m}\times\mathbb{R}^{n}\rightarrow\mathbb{R}^{m}$
as
$$T_{2}^{n}(v,x)=x-\gamma_{n}\nabla f_{2}(x)-\lambda_{n} D^{T}\circ T^{n}_{1}.\eqno{(3.11)}$$

 Denote
$$T^{n}(v,x)=(T^{n}_{1}(v,x),T^{n}_{2}(v,x)).\eqno{(3.12)}$$
\par
In the following, we will show the  algorithm PDFP$^{2}O_{DS_{n}}$
is a  modified Mann iterative method related to the operator $S^{n}$. \\
\begin{thm}
Suppose
$0<\liminf_{n\rightarrow\infty}\alpha_{n}\leq\limsup_{n\rightarrow\infty}\alpha_{n}<1$.
Set $S^{n}=\alpha_{n}I+(1-\alpha_{n})T^{n}$. Then the sequence
$u_{n}$ of $S^{n}$ is exactly the one obtained by the algorithm
PDFP$^{2}O_{DS_{n}}$.
\end{thm}
\begin{proof}
According to the definitions in (3.10)-(3.12), the component form of
$u_{n+1} = T^{n} (u_{n} )$ can be expressed as
$$
\left\{
\begin{array}{l}
v_{n+1}=T^{n}_{1}(v_{n},x_{n})=(I-prox_{\frac{\gamma_{n}}{\lambda_{n}}f_{1}})(D(x_{n}-\gamma_{n} \nabla f_{2}(x_{n}))+(I-\lambda_{n} DD^{T})v_{n}),\\
x_{n+1}=T^{n}_{2}(v_{n},x_{n})=x_{n}-\gamma_{n}
\nabla f_{2}(x_{n})-\lambda_{n} D^{T}\circ T_{1}^{n}(v_{n},x_{n})\\
=x_{n}-\gamma _{n}\nabla f_{2}(x_{n})-\lambda_{n} D^{T}v_{n+1}.
\end{array}
\right.
$$
Therefore, the iteration $u_{n+1} = T ^{n}(u_{n} )$ is equivalent to
(2.2). Employing the similar argument, we can obtain the conclusion
for general $S^{n}$ with
$0<\liminf_{n\rightarrow\infty}\alpha_{n}\leq\limsup_{n\rightarrow\infty}\alpha_{n}<1$.
\end{proof}
\begin{rmk}
From the last result, we find out that algorithm
PDFP$^{2}O_{DS_{n}}$ can also be obtained in the setting of fixed
point iteration immediately.
\end{rmk}

\begin{thm}
 Let $T^{n}$, $T$
be defined by 3.12, 3.5  respectively, suppose
$0<\liminf_{n\rightarrow\infty}\gamma_{n}\leq\limsup_{n\rightarrow\infty}\gamma_{n}<2\beta$,
$0<\liminf_{n\rightarrow\infty}\lambda_{n}\leq\limsup_{n\rightarrow\infty}\lambda_{n}\leq1/\lambda_{\max}(DD^{T})$,
if for any bounded sequence
$\{u_{n}\}\subset\mathbb{R}^{m}\times\mathbb{R}^{n}$,
$$\lim_{n\rightarrow\infty}\|u_{n}-T^{n}(u_{n})\|_{\lambda}=0,$$ then
there exists a subsequence $\{u_{n_{k}}\}\subset\{u_{n}\}$ such that
$\lim_{n_{k}\rightarrow\infty}\|u_{n_{k}}-T(u_{n_{k}})\|_{\lambda}=0$
.
\end{thm}
\begin{proof}
Since the sequence $\gamma_{n}$ is bounded, there exists a
subsequence $\gamma_{n_{k}}\subset\gamma_{n}$ such that
$\gamma_{n_{k}}\rightarrow \gamma$ with $\gamma\in (0,2\beta)$.
Since the sequence $\lambda_{n}$ is bounded, there exists a
subsequence $\lambda_{n_{k}}\subset\lambda_{n}$ such that
$\lambda_{n_{k}}\rightarrow \lambda$ with $\lambda\in
(0,1/\lambda_{max}(DD^{T})]$. Let $T$ be defined by 3.5, since
$\gamma\in (0,2\beta)$ and $\lambda\in (0,1/\lambda_{max}(DD^{T})]$,
so $T$ is a
 nonexpansive mapping under the norm $\|\cdot\|_{\lambda}$. Since
sequence $\{u_{n}\}$ is bounded and
$\lim_{n\rightarrow\infty}\|u_{n}-T^{n}(u_{n})\|_{\lambda}=0$.\\
 We can know that
\begin{align*}
\|u_{n_{k}}-T(u_{n_{k}})\|_{\lambda}&\leq\|u_{n_{k}}-T^{n_{k}}(u_{n_{k}})\|_{\lambda}+\|T^{n_{k}}u_{n_{k}}-T(u_{n_{k}})\|_{\lambda}.\tag{3.13}
\end{align*}
From (3.9) we know
\begin{align*}
\|T^{n_{k}}u_{n_{k}}-T(u_{n_{k}})\|_{\lambda}^{2}&=\|T_{1}^{n_{k}}u_{n_{k}}-T_{1}(u_{n_{k}})\|^{2}+\lambda\|T_{2}^{n_{k}}u_{n_{k}}-T_{2}(u_{n_{k}})\|^{2}.\tag{3.14}
\end{align*}
By lemma 3.2, $I-prox_{\frac{\gamma_{n}}{\lambda_{n}}f_{1}}$ is a
firmly nonexpansive operator. So
\begin{align*}
\|T_{1}^{n_{k}}u_{n_{k}}-T_{1}(u_{n_{k}})\|&=\|(I-prox_{\frac{\gamma_{n_{k}}}{\lambda_{n_{k}}}f_{1}})(D(x_{n_{k}}-\gamma_{n_{k}}
\nabla f_{2}(x_{n_{k}}))\\
&+(I-\lambda_{n_{k}}
DD^{T})v_{n_{k}})-(I-prox_{\frac{\gamma}{\lambda}f_{1}})(D(x_{n_{k}}-\gamma
\nabla f_{2}(x_{n_{k}}))\\
&+(I-\lambda DD^{T})v_{n_{k}})\| \\
&=\|prox_{\frac{\gamma_{n_{k}}}{\lambda_{n_{k}}}f_{1}}(D(x_{n_{k}}-\gamma_{n_{k}}
\nabla f_{2}(x_{n_{k}}))\\
&+(I-\lambda_{n_{k}}
DD^{T})v_{n_{k}})-prox_{\frac{\gamma}{\lambda}f_{1}}(D(x_{n_{k}}-\gamma
\nabla f_{2}(x_{n_{k}}))\\
&+(I-\lambda DD^{T})v_{n_{k}})\| \\
&\leq\|prox_{\frac{\gamma_{n_{k}}}{\lambda_{n_{k}}}f_{1}}(D(x_{n_{k}}-\gamma_{n_{k}}
\nabla f_{2}(x_{n_{k}}))\\
&+(I-\lambda_{n_{k}}
DD^{T})v_{n_{k}})-prox_{\frac{\gamma}{\lambda}f_{1}}(D(x_{n_{k}}-\gamma_{n_{k}}
\nabla f_{2}(x_{n_{k}}))\\
&+(I-\lambda_{n_{k}}DD^{T})v_{n_{k}})+prox_{\frac{\gamma}{\lambda}f_{1}}(D(x_{n_{k}}-\gamma_{n_{k}}
\nabla f_{2}(x_{n_{k}}))\\
&+(I-\lambda_{n_{k}}DD^{T})v_{n_{k}})-prox_{\frac{\gamma}{\lambda}f_{1}}(D(x_{n_{k}}-\gamma
\nabla f_{2}(x_{n_{k}}))\\
&+(I-\lambda DD^{T})v_{n_{k}})\| \\
&\leq\|prox_{\frac{\gamma_{n_{k}}}{\lambda_{n_{k}}}f_{1}}(D(x_{n_{k}}-\gamma_{n_{k}}
\nabla f_{2}(x_{n_{k}}))\\
&+(I-\lambda_{n_{k}}
DD^{T})v_{n_{k}})-prox_{\frac{\gamma}{\lambda}f_{1}}(D(x_{n_{k}}-\gamma_{n_{k}}
\nabla f_{2}(x_{n_{k}}))\\
&+(I-\lambda_{n_{k}}DD^{T})v_{n_{k}})\|+\|prox_{\frac{\gamma}{\lambda}f_{1}}(D(x_{n_{k}}-\gamma_{n_{k}}
\nabla f_{2}(x_{n_{k}}))\\
&+(I-\lambda_{n_{k}}DD^{T})v_{n_{k}})-prox_{\frac{\gamma}{\lambda}f_{1}}(D(x_{n_{k}}-\gamma
\nabla f_{2}(x_{n_{k}}))\\
&+(I-\lambda DD^{T})v_{n_{k}})\|.\tag{3.15}
\end{align*}
Let $z_{n_{k}}=D(x_{n_{k}}-\gamma_{n_{k}} \nabla
f_{2}(x_{n_{k}}))+(I-\lambda_{n_{k}} DD^{T})v_{n_{k}}$. Since
$J_{\lambda \partial f_{1}}=(I+\lambda\partial
f_{1})^{-1}=prox_{\lambda f_{1}}$ and by lemma 3.3, we can know
$prox_{\nu f_{1}}x=prox_{\mu
f_{1}}(\frac{\mu}{\nu}x+(1-\frac{\mu}{\nu})prox_{\nu f_{1}}x)$, so
\begin{align*}
\|prox_{\frac{\gamma_{n_{k}}}{\lambda_{n_{k}}}f_{1}}(z_{n_{k}})-prox_{\frac{\gamma}{\lambda}f_{1}}(z_{n_{k}})\|
&=\|prox_{\frac{\gamma_{n_{k}}}{\lambda_{n_{k}}}f_{1}}(z_{n_{k}})-prox_{\frac{\gamma}{\lambda}f_{1}}(z_{n_{k}})\|\\
&=\|prox_{\frac{\gamma}{\lambda}f_{1}}(\frac{\gamma}{\lambda}/\frac{\gamma_{n_{k}}}{\lambda_{n_{k}}}z_{n_{k}}+(1-\frac{\gamma}{\lambda}/\frac{\gamma_{n_{k}}}{\lambda_{n_{k}}})prox_{\frac{\gamma_{n_{k}}}{\lambda_{n_{k}}}f_{1}}z_{n_{k}})\\
&-prox_{\frac{\gamma}{\lambda}f_{1}}(z_{n_{k}})\|\\
&\leq\|\frac{\gamma}{\lambda}/\frac{\gamma_{n_{k}}}{\lambda_{n_{k}}}z_{n_{k}}+(1-\frac{\gamma}{\lambda}/\frac{\gamma_{n_{k}}}{\lambda_{n_{k}}})prox_{\frac{\gamma_{n_{k}}}{\lambda_{n_{k}}}f_{1}}z_{n_{k}}-z_{n_{k}}\|\\
&=|1-\frac{\gamma}{\lambda}/\frac{\gamma_{n_{k}}}{\lambda_{n_{k}}}|\|prox_{\frac{\gamma_{n_{k}}}{\lambda_{n_{k}}}f_{1}}z_{n_{k}}-z_{n_{k}}\|.\tag{3.16}
\end{align*}
On the other hand

\begin{align*}
&\|prox_{\frac{\gamma}{\lambda}f_{1}}(D(x_{n_{k}}-\gamma_{n_{k}}
\nabla f_{2}(x_{n_{k}}))
+(I-\lambda_{n_{k}}DD^{T})v_{n_{k}})\\
&-prox_{\frac{\gamma}{\lambda}f_{1}}(D(x_{n_{k}}-\gamma \nabla
f_{2}(x_{n_{k}}))+(I-\lambda
DD^{T})v_{n_{k}})\|\\
&\leq\|(\gamma-\gamma_{n_{k}})D\nabla
f_{2}(x_{n_{k}})+(\lambda-\lambda_{n_{k}})DD^{T}v_{n_{k}}\|\\
&\leq|\gamma-\gamma_{n_{k}}|\|D\nabla
f_{2}x_{n_{k}}\|+|\lambda-\lambda_{n_{k}}|\|DD^{T}v_{n_{k}}\|.\tag{3.17}
\end{align*}
Put (3.16) and (3.17) into (3.15), we can know
\begin{align*}
\|T_{1}^{n_{k}}u_{n_{k}}-T_{1}(u_{n_{k}})\|&\leq|1-\frac{\gamma}{\lambda}/\frac{\gamma_{n_{k}}}{\lambda_{n_{k}}}|\|prox_{\frac{\gamma_{n_{k}}}{\lambda_{n_{k}}}f_{1}}z_{n_{k}}-z_{n_{k}}\|\\
&+|\gamma-\gamma_{n_{k}}|\|D\nabla
f_{2}x_{n_{k}}\|+|\lambda-\lambda_{n_{k}}|\|DD^{T}v_{n_{k}}\|.\tag{3.18}
\end{align*}
Since $\gamma_{n_{k}}\rightarrow \gamma$ and
$\lambda_{n_{k}}\rightarrow \lambda$, from (3.18) we can know
$$\|T_{1}^{n_{k}}u_{n_{k}}-T_{1}(u_{n_{k}})\|\rightarrow0.\eqno{(3.19)}$$
 It follows from (3.11) that
\begin{align*}
\|T_{2}^{n_{k}}u_{n_{k}}-T_{2}(u_{n_{k}})\|&=\|x_{n_{k}}-\gamma_{n_{k}}
\nabla
f_{2}(x_{n_{k}})-\lambda_{n_{k}}D^{T}T_{1}^{n_{k}}\\
&-x_{n_{k}}-\gamma \nabla
f_{2}(x_{n_{k}})+\lambda D^{T}T_{1}\|\\
 &\leq|\gamma-\gamma_{n_{k}}|\|\nabla
f_{2}x_{n_{k}}\|+|\lambda-\lambda_{n_{k}}|\|D^{T}T_{1}^{n_{k}}\|\\
&+\|\lambda D^{T}\|\|T_{1}^{n_{k}}-T_{1}\|.\tag{3.20}
\end{align*}
Since $\gamma_{n_{k}}\rightarrow \gamma$ and
$\lambda_{n_{k}}\rightarrow \lambda$, from (3.19) we can know
$$\|T_{2}^{n_{k}}u_{n_{k}}-T_{2}(u_{n_{k}})\|\rightarrow0.\eqno{(3.21)}$$
Put (3.19) and (3.21) into (3.14), we can know
$$\|T^{n_{k}}u_{n_{k}}-T(u_{n_{k}})\|_{\lambda}^{2}\rightarrow0.\eqno{(3.22)}$$
Put  (3.22) into (3.13), we can know
$$\|u_{n_{k}}-T(u_{n_{k}})\|_{\lambda}\rightarrow0.\eqno{(3.23)}$$

\end{proof}

\begin{thm}
 Let $T^{n}$, $T$
be defined by 3.12, 3.5  respectively , suppose
$0<\liminf_{n\rightarrow\infty}\gamma_{n}\leq\limsup_{n\rightarrow\infty}\gamma_{n}<2\beta$,
$0<\liminf_{n\rightarrow\infty}\lambda_{n}\leq\limsup_{n\rightarrow\infty}\lambda_{n}\leq1/\lambda_{\max}(DD^{T})$,
 let $u_{n}$ be  sequence defined by PDFP$^{2}O_{DS_{n}}$, that is:
$$u_{n+1}=S^{n}(u_{n})=\alpha_{n}u_{n}+(1-\alpha_{n})T^{n}u_{n},\eqno{(3.24)}$$
where $\alpha_{n}$ satisfy
$$0<\liminf_{n\rightarrow\infty}\alpha_{n}\leq\limsup_{n\rightarrow\infty}\alpha_{n}<1.\eqno{(3.25)}$$
Then the sequence $\{u_{n}\}$ defined by (3.24) converges  to a fixed
point of $T$, and the sequence $\{x_{n}\}$ converges to a solution
of problem (1.1).

\end{thm}

\begin{proof}

Let $\hat{u}=(\hat{v},\hat{x})\in
\mathbb{R}^{m}\times\mathbb{R}^{n}$ be a fixed point of $T$. From
(3.24) and lemma 3.4, we have
\begin{align*}
\|u_{n+1}-\hat{u}\|_{\lambda}^{2}&=\|(1-\alpha_{n})u_{n}+\alpha_{n}T^{n}u_{n}-\hat{u}\|_{\lambda}^{2}\\
&=\alpha_{n}\|u_{n}-\hat{u}\|_{\lambda}^{2}+(1-\alpha_{n})\|T^{n}u_{n}-\hat{u}\|_{\lambda}^{2}-\alpha_{n}(1-\alpha_{n})\|u_{n}-T^{n}u_{n}\|_{\lambda}^{2}
.\tag{3.26}
\end{align*}
Since the sequence $\lambda_{n}$ is bounded, there exists a
convergent subsequence converges to $\lambda$,  without loss of
generality, we may assume that the convergent subsequence is
$\lambda_{n}$ itself, then we have $\lambda_{n}\rightarrow \lambda$.
That is, $\exists N_{0}\in N$
 such that  $\lambda_{n}\leq\lambda$. So by the similar  proof of  theorem 3.3 in
[1], for $\forall n\geq N_{0}$, we have\\
\begin{align*}
\|T^{n}u_{n}-\hat{u}\|_{\lambda}^{2}&=\|T^{n}u_{n}-T^{n}\hat{u}\|_{\lambda}^{2}\\
&\leq\|u_{n}-\hat{u}\|_{\lambda}^{2}+
\lambda_{n}\||v_{n}-\hat{v}\|^{2}+(\lambda-\lambda_{n})\|T^{n}_{1}(u_{n})-\hat{v}\|^{2}\\
&\leq\|u_{n}-\hat{u}\|_{\lambda}^{2}.\tag{3.27}
\end{align*}
Substituting (3.27) into (3.26), we obtain
\begin{align*}
\|u_{n+1}-\hat{u}\|_{\lambda}^{2}&\leq\|u_{n}-\hat{u}\|_{\lambda}^{2}-\alpha_{n}(1-\alpha_{n})\|u_{n}-T^{n}u_{n}\|_{\lambda}^{2}
.\tag{3.28}
\end{align*}
Which implies that
$$\|u_{n+1}-\hat{u}\|_{\lambda}\leq\|u_{n}-\hat{u}\|_{\lambda},$$
 this implies that sequence $u_{n}$ is a Fej\'{e}r  monotone
sequence, and
$\lim_{n\rightarrow\infty}\|u_{n+1}-\hat{u}\|_{\lambda}$ exists.

Since the sequence $\alpha_{n}$ satisfies (3.25), there exists
$\bar{a}, \underline{a}\in (0,1) $ such that
$\underline{a}<\alpha_{n} <\bar{a}.$ So by (3.28), we know
\begin{align*}
\underline{a}(1-\bar{a})\|u_{n}-T^{n}u_{n}\|_{\lambda}^{2}&\leq\alpha_{n}(1-\alpha_{n})\|u_{n}-T^{n}u_{n}\|_{\lambda}^{2}\\
&\leq\|u_{n}-\hat{u}\|_{\lambda}^{2}-\|u_{n+1}-\hat{u}\|_{\lambda}^{2},\tag{3.29}
\end{align*}
Let $n\rightarrow\infty$ in (3.29), we have
$$\|u_{n}-T^{n}u_{n}\|_{\lambda}\rightarrow0.\eqno{(3.30)}$$

Since the sequence $u_{n}$ is bounded and there exists a convergent
subsequence $u_{n_{j}}$ such that
$$u_{n_{j}}\rightarrow \tilde{u},\eqno{(3.31)}$$
for some $\tilde{u}\in \mathbb{R}^{m}\times\mathbb{R}^{n}$.\\
 From Theorem 3.2 and
(3.30), we have

$$\|u_{n_{j}}-Tu_{n_{j}}\|_{\lambda}\rightarrow0.$$ By Lemma
3.5, we know $\tilde{u}\in Fix(T)$. Moreover, we know that
$\|u_{n}-\hat{u}\|_{\lambda}$ is non-increasing for any fixed point
$\hat{u}$ of $T$. In particular, by choosing $\hat{u}=\tilde{u}$, we
have $\|u_{n}-\tilde{u}\|_{\lambda}$ is non-increasing. Combining
this and (3.31) yields
$$u_{n}\rightarrow \tilde{u}.$$

Writing $\tilde{u}=(\tilde{v},\tilde{x})$ with $\tilde{v}\in
\mathbb{R}^{m}, \tilde{x}\in\mathbb{R}^{n}$, we find from Lemma 3.6
that $\tilde{x}$ is the solution of problem (1.1).
\end{proof}
\subsection{Linear convergence rate for special cases}
In this section, we will give some stronger theoretical results
about the convergence rate in some special cases. For this, we
present the following condition.
\begin{Cond}
For any two real numbers $\lambda$ and $\gamma$ satisfying that
$0<\gamma<2\beta$ and $0<\lambda\leq1/\lambda_{\max}(DD^{T})$, there
exist $\mu, \nu\in[0,1)$ such that $\|I-\lambda
DD^{T}\|_{2}\leq\mu^{2}$ and
$$\|g(x)-g(y)\|_{2}\leq\nu\|x-y\|_{2}, ~~~~for~ all ~~x,y\in \mathbb{R}^{n}.$$

\end{Cond}
\begin{rmk}
If $D$ has full row rank, $f_{2}$ is strongly convex, i.e. there
exists some $\sigma > 0$ such that
$$\langle\nabla f_{2}(x)-\nabla f_{2}(y),x-y\rangle\geq\sigma\|x-y\|_{2}^{2}, ~~~~for~ all ~~x,y\in \mathbb{R}^{n},\eqno{(3.33)}$$
then this condition can be satisfied. In fact, when $D$ has a full
row rank, we can choose
$$\mu^{2}=1-\lambda\lambda_{\min}(DD^{T})$$
where $\lambda_{\min}(DD^{T})$ denotes the smallest eigenvalue of
$DD^{T}$ . In this case, $\mu^{2}$ takes its minimum
$$(\mu^{2})_{\min}=1-\frac{\lambda_{\min}(DD^{T})}{\lambda_{\max}(DD^{T})}$$
at $\lambda=1/\lambda_{\max}(DD^{T})$. On the other hand, since
$f_{2}$ have $1/\beta$-Lipschitz continuous gradient and is strongly
convex, it follows from proof in [1] we know
\begin{align*}
\|g(x)-g(y)\|_{2}^{2}
&\leq(1-(\frac{\gamma\sigma(2\beta-\gamma)}{\beta}))\|x-y\|_{2}^{2}.\tag{3.29}
\end{align*}
Hence we can choose
$$\nu^{2}=1-(\frac{\gamma\sigma(2\beta-\gamma)}{\beta}).$$
In particular, if we choose $\beta=\gamma$, then $\nu^{2}$ takes its
minimum in the present form:
$$\nu^{2}=1-\sigma\gamma.$$
\end{rmk}
Despite most of our interesting problems not belonging to these
special cases, and there will be more efficient algorithms if
condition 3.1 is satisfied, the following results still have some
theoretical values where the best performance of
PDFP$^{2}O_{DS_{n}}$ can be achieved. First of all, we show that $S$
is contractive under condition 3.1.

\begin{thm}
Assume condition 3.1 holds true. Let the operator $T$ be given in
(3.5) and $S=\alpha_{n}I+(1-\alpha_{n})T$ for
$0<\liminf_{n\rightarrow\infty}\alpha_{n}\leq\limsup_{n\rightarrow\infty}\alpha_{n}<1$.
Then $S$ is contractive under the norm $\|\cdot\|_{\lambda}$.
\end{thm}

\begin{proof}
Let $\eta=\max\{\mu,\nu\}$. It is clear that $0 \leq\eta\leq 1$.
Then, owing to the condition 3.1 and the proof of Theorem 3.6 of
[1], for all $u_{1}=(v_{1},x_{1}),
u_{2}=(v_{2},x_{2})\in\mathbb{R}^{m}\times\mathbb{R}^{n}$, there
holds
$$\|T(u_{1})-T(u_{2})\|_{\lambda}\leq\eta\|u_{1}-u_{2}\|_{\lambda},$$
then
$$\|S(u_{1})-S(u_{2})\|_{\lambda}\leq\alpha_{n}\|u_{1}-u_{2}\|_{\lambda}+(1-\alpha_{n})\|T(u_{1})-T(u_{2})\|_{\lambda}\leq\theta_{\alpha_{n}}\|u_{1}-u_{2}\|_{\lambda},$$
with $\theta_{\alpha_{n}}=\alpha_{n}+(1-\alpha_{n})\eta\in(0,1)$.
So, operator $S$ is contractive. By the Banach contraction mapping
theorem, it has a unique fixed point, denoted by
$\bar{u}=(\bar{v},\bar{x})$. It is obvious that $S$ has the same
fixed points as $T$, so $\bar{x}$ is the unique solution of problem
(1.1) from lemma 3.6.

\end{proof}

 Now, we are ready to analyze the convergence rate of PDFP$^{2}O_{DS_{n}}$ .

\begin{thm}
Assume condition 3.1 holds true. Let the operator $T$ be given in
(3.5) and $T^{n}$ be defined as 3.12 with $\emptyset\neq
Fix(T)=\bigcap_{n=1}^{\infty}Fix(T^{n})$. For any
$u_{0}\in\mathbb{R}^{m}\times\mathbb{R}^{n}$, the sequence $u_{n}$
be a sequence obtained by algorithm PDFP$^{2}O_{DS_{n}}$, and
$0<\liminf_{n\rightarrow\infty}\alpha_{n}\leq\limsup_{n\rightarrow\infty}\alpha_{n}<1$.
Then the sequence $\{u_{n}\}$ must converge to the unique fixed
point
$\bar{u}=(\bar{v},\bar{x})\in\mathbb{R}^{m}\times\mathbb{R}^{n}$ of
$T$ with $\bar{x}$ being the unique solution of problem (1.1).
Furthermore, there holds the estimate
$$\|x_{n}-\bar{x}\|_{2}\leq \frac{d(\theta_{\alpha_{n}})^{n}}{1-\theta_{\alpha_{n}}},\eqno{(3.34)}$$
where $d=\|u_{1}-u_{0}\|_{\lambda}$,
$\theta_{\alpha_{n}}=\alpha_{n}+(1-\alpha_{n})\eta\in(0,1) $ and
$\eta=\max\{\mu,\nu\}$ with $\mu$ and $\nu$ given in condition 3.1.

\end{thm}

\begin{proof}
From Theorem 3.3, we can know that the sequence $\{u_{n}\}$
converges to $\bar{u}$. On the other hand, it follows from theorem
3.4 that
$$\|u_{n+1}-u_{n}\|_{\lambda}\leq\theta_{\alpha_{n}}\|u_{n}-u_{n-1}\|_{\lambda}\leq\cdots\leq(\theta_{\alpha_{n}})^{n}\|u_{1}-u_{0}\|_{\lambda}=d(\theta_{\alpha_{n}})^{n}.$$
So for all $0 < l \in \mathbb{N} $,
$$\|u_{n+l}-u_{n}\|_{\lambda}\leq\sum_{i=1}^{l}\|u_{n+i}-u_{n+i-1}\|_{\lambda}=d(\theta_{\alpha_{n}})^{n}\sum_{i=1}^{l}(\theta_{\alpha_{n}})^{i-1}\leq \frac{d(\theta_{\alpha_{n}})^{n}}{1-\theta_{\alpha_{n}}},$$
which immediately implies
$$\|x_{n}-\bar{x}\|_{2}\leq\|u_{n}-\bar{u}\|_{\lambda}\leq \frac{d(\theta_{\alpha_{n}})^{n}}{1-\theta_{\alpha_{n}}},$$
by letting $l \rightarrow+\infty$. The desired estimate (3.29) is
then obtained.
\end{proof}

\begin{rmk}
Since sequence $\alpha_{n}$ satisfy
$0<\liminf_{n\rightarrow\infty}\alpha_{n}\leq\limsup_{n\rightarrow\infty}\alpha_{n}<1$,
then exists $\underline{a}, \bar{a}\in(0,1)$ such that
$\underline{a}< \alpha_{n}<\bar{a}$. So we have
$\underline{a}+(1-\bar{a})\eta<\alpha_{n}+(1-\alpha_{n})\eta$.\\
 In
particular, if we choose
$\theta_{\alpha_{n}}=\underline{a}+(1-\bar{a})\eta=\theta_{a}$, then
we obtain
$$\|x_{n}-\bar{x}\|_{2}\leq \frac{d(\theta_{a})^{n}}{1-\theta_{a}}.\eqno{(3.35)}$$
It will follow that our scheme shows an
$o(\frac{d(\theta_{a})^{n}}{1-\theta_{a}})$ convergence to the
optimum for the variable $x_{n}$, which is an optimal rate.

\end{rmk}

\section{Connections to other algorithms}

We will further investigate the proposed algorithm PDFP$^{2}O_{DS}$
from the perspective of primal-dual forms and establish the
connections to other existing methods.
\subsection{Primal-dual and proximal point algorithms}
For problem (1.1), we can write its primal-dual form using the
Fenchel duality [18] as
$$\min_{x}\max_{y}G(x,v):=\langle Dx,v\rangle-f_{1}^{\ast}(v)+f_{2}(x),\eqno{(4.1)}$$
where $f_{1}^{\ast}$ is the convex conjugate function of $f_{1}$
defined by
$$f_{1}^{\ast}(v)=\sup_{w\in\mathbb{R}^{m}}\langle v, w\rangle-f_{1}(v).$$
\par
By introducing a new intermediate variable $y_{n+1}$, equations
(2.2) are reformulated as
$$
\left\{
\begin{array}{l}
y_{n+1}=x_{n}-\gamma_{n}\nabla f_{2}(x_{n})-\lambda_{n}D^{T}v_{n},~~(4.2a)\\
v_{n+1}=(I-prox_{\frac{\gamma_{n}}{\lambda_{n}}f_{1}})(Dy_{n+1}+v_{n}),~(4.2b)\\
x_{n+1}=x_{n}-\gamma_{n}\nabla
f_{2}(x_{n})-\lambda_{n}D^{T}v_{n+1}.(4.2c)
\end{array}
\right.
$$
According to Moreau decomposition (see equation (2.21) in [2]), for
all $v \in \in\mathbb{R}^{m}$, we have
$$v=v_{\frac{\gamma_{n}}{\lambda_{n}}}^{\oplus}+v_{\frac{\gamma_{n}}{\lambda_{n}}}^{\ominus},$$
where
$v_{\frac{\gamma_{n}}{\lambda_{n}}}^{\oplus}=prox_{\frac{\gamma_{n}}{\lambda_{n}}f_{1}}v$,
$v_{\frac{\gamma_{n}}{\lambda_{n}}}^{\ominus}=\frac{\gamma_{n}}{\lambda_{n}}prox_{\frac{\gamma_{n}}{\lambda_{n}}f_{1}^{\ast}}(\frac{\lambda_{n}}{\gamma_{n}}v)$,
from which we know
$$(I-prox_{\frac{\gamma_{n}}{\lambda_{n}}f_{1}})(Dy_{n+1}+v_{n})
=\frac{\gamma_{n}}{\lambda_{n}}prox_{\frac{\gamma_{n}}{\lambda_{n}}f_{1}^{\ast}}(\frac{\lambda_{n}}{\gamma_{n}}Dy_{n+1}+\frac{\lambda_{n}}{\gamma_{n}}v_{n}).$$
Let  $\bar{v}_{n}=\frac{\lambda_{n}}{\gamma_{n}}v_{n}$. Then (4.2)
can be reformulated as
$$
\left\{
\begin{array}{l}
y_{n+1}=x_{n}-\gamma_{n}\nabla f_{2}(x_{n})-\gamma_{n}D^{T}\bar{v}_{n},~~~(4.3a)\\
v_{n+1}=prox_{\frac{\gamma_{n}}{\lambda_{n}}f_{1}^{\ast}}(\frac{\lambda_{n}}{\gamma_{n}}Dy_{n+1}+\bar{v}_{n}),~~~~~~(4.3b)\\
x_{n+1}=x_{n}-\gamma_{n}\nabla
f_{2}(x_{n})-\gamma_{n}D^{T}\bar{v}_{n+1}.~(4.3c)
\end{array}
\right.
$$

For terms of the saddle point formulation (4.1),  with the same idea
in [1](4.1 Primal-dual and proximal point algorithms), the
iterations (4.3) can be expressed as

$$
\left\{
\begin{array}{l}
\bar{v}_{n+1}=\arg\max_{\bar{v}\in\mathbb{R}^{m}}G(x_{n+1},\bar{v})-\frac{\gamma_{n}}{2\lambda_{n}}\|\bar{v}-\bar{v}_{n}\|^{2}_{M_{n}},(4.4a)\\
x_{n+1}=x_{n}-\gamma_{n}\nabla_{x}G(x_{n},\bar{v}_{n+1}),~~~~~~~~~~~~~~~~~~~~~~~~(4.4b)
\end{array}
\right.
$$
where $M_{n}=I-\lambda_{n}DD^{T}$.

Table 1.  Comparison between CP ($\theta_{n} = 1$) and
PDFP$^{2}O_{DS}$.

\begin{tabular}{llll}
\\
\hline\noalign{\smallskip}
\emph{} &\emph{CP($\theta_{n}=1$)} &   & \\
\noalign{\smallskip}\hline\noalign{\smallskip} Form &
$\bar{v}_{n+1}=(I+\sigma_{n}\partial
f_{1}^{\ast})^{-1}(\bar{v}_{n}+\sigma_{n}
Dy_{n+1})$ &\\
     & $x_{n+1}=(I+\tau_{n}\nabla f_{2})^{-1}(x_{n}-\tau_{n}D^{T}\bar{v}_{n+1})$&\\
     & $y_{n+1}=2x_{n+1}-x_{n}$&\\
Convergence & $0<\liminf_{n\rightarrow\infty}\sigma_{n}\tau_{n}\leq\limsup_{n\rightarrow\infty}\sigma_{n}\tau_{n}<1/\lambda_{\max}(DD^{T})$ &  \\
\noalign{\smallskip}\hline
\emph{} &\emph{PDFP$^{2}O_{DS}$} &   & \\
\noalign{\smallskip}\hline\noalign{\smallskip} Form &
$\bar{v}_{n+1}=(I+\frac{\lambda_{n}}{\gamma_{n}}\partial
f_{1}^{\ast})^{-1}(\bar{v}_{n}+\frac{\lambda_{n}}{\gamma_{n}}
Dy_{n+1})$ &\\
     & $x_{n+1}=x_{n}-\gamma_{n}\nabla f_{2}(x_{n})-\gamma_{n}D^{T}\bar{v}_{n+1}$&\\
     & $y_{n+1}=x_{n+1}-\gamma_{n}\nabla f_{2}(x_{n+1})-\gamma_{n}D^{T}\bar{v}_{n+1}$&\\
Convergence &
$0<\liminf_{n\rightarrow\infty}\gamma_{n}\leq\limsup_{n\rightarrow\infty}\gamma_{n}<2\beta$&  \\
 &      $0<\liminf_{n\rightarrow\infty}\lambda_{n}\leq\limsup_{n\rightarrow\infty}\lambda_{n}\leq1/\lambda_{\max}(DD^{T})$ &  \\
Relation & $\sigma_{n}=\frac{\lambda_{n}}{\gamma_{n}}$, $\tau_{n}=\gamma_{n}$\\
\noalign{\smallskip}\hline
\end{tabular}\\
\par
This leads to a close connection with a class of primal-dual method
studied in [19-22]. For example, in [19], Chambolle and Pock
proposed the following scheme for solving (4.1):\\
$$
\left\{
\begin{array}{l}
\bar{v}_{n+1}=(I+\sigma_{n}\partial
f_{1}^{\ast})^{-1}(\bar{v}_{n}+\sigma_{n}
Dy_{n+1}),(4.5a)\\
x_{n+1}=(I+\tau_{n}\nabla f_{2})^{-1}(x_{n}-\tau_{n}D^{T}\bar{v}_{n+1}),(4.5b)\\
y_{n+1}=\theta_{n}x_{n+1}-x_{n},(4.5c)
\end{array}
\right.
$$
where $\sigma_{0}, \tau_{0} > 0$, $\theta_{n} \in [0, 1]$ is a
variable relaxation parameter. For $\sigma_{n}=\sigma$,
$\tau_{n}=\tau$ and  $\theta_{n}\equiv0$, we can obtain the
classical Arrow-Hurwicz-Uzawa (AHU) method in [23]. The convergence
of AHU with very small step length is shown in [20]. Under some
assumptions on $f_{1}$ or strong convexity of $f_{2}$, global
convergence of the primal-dual gap can also be shown with specific
chosen adaptive steplength [19].
\par
According to equation (4.3), using the relation
$prox_{\frac{\gamma_{n}}{\lambda_{n}}f_{1}^{\ast}} =
(I+\frac{\lambda_{n}}{\gamma_{n}}\partial f_{1}^{\ast})^{-1}$, and
changing the order of these equations, we know that PDFP$^{2}O_{DS}$
is equivalent to
$$
\left\{
\begin{array}{l}
\bar{v}_{n+1}=(I+\frac{\lambda_{n}}{\gamma_{n}}\partial
f_{1}^{\ast})^{-1}(\bar{v}_{n}+\frac{\lambda_{n}}{\gamma_{n}}
Dy_{n}),~~~~~~(4.6a)\\
x_{n+1}=x_{n}-\gamma_{n}\nabla f_{2}(x_{n})-\gamma_{n}D^{T}\bar{v}_{n+1},~~~~~(4.6b)\\
y_{n+1}=x_{n+1}-\gamma_{n}\nabla
f_{2}(x_{n+1})-\gamma_{n}D^{T}\bar{v}_{n+1}.(4.6c)
\end{array}
\right.
$$
Let $\sigma_{n}=\frac{\lambda_{n}}{\gamma_{n}}$,
$\tau_{n}=\gamma_{n}$ $(n\in\mathbb{N})$, then we can see that
equations (4.5b) and (4.5c) are approximated by two explicit steps
(4.6b)-(4.6c). In summary, we list the comparisons of CP for
$\theta_{n}\equiv1$ with the fixed step length and PDFP$^{2}O_{DS}$
in table 1.
\subsection{Splitting type of methods}
\par
There are other types of methods which are designed to solve problem
(1.1) based on the notion of an augmented Lagrangian. For
simplicity, we only study the connections and differences in
alternating split Bregman (ASB), split inexact Uzawa (SIU) and
PDFP$^{2}O_{DS}$, for $f_{2}(x)=\frac{1}{2}\|Ax-b\|^{2}_{2}$.
\par
ASB present by Goldstein and Osher [24] can be described as follows:
$$
\left\{
\begin{array}{l}
x_{n+1}=(A^{T}A+\nu_{n}D^{T}D)^{-1}(A^{T}b+\nu_{n}D^{T}(d_{n}-v_{n})),(4.7a)\\
d_{n+1}=prox_{\frac{1}{\nu_{n}}f_{1}}(Dx_{n+1}+v_{n}),~~~~~~~~~~~~~~~~~~~~~~~~~~~(4.7b)\\
v_{n+1}=v_{n}-(d_{n+1}-Dx_{n+1}),~~~~~~~~~~~~~~~~~~~~~~~~~~~~~(4.7c)
\end{array}
\right.
$$
where $\liminf_{n\rightarrow\infty}\nu_{n} > 0$ is a dynamic
parameter. The explicit SIU method proposed in the literature [22]
can be described as
$$
\left\{
\begin{array}{l}
x_{n+1}=x_{n}-\delta_{n}A^{T}(Ax_{n}-b)-\delta_{n}\nu_{n}D^{T}(Dx_{n}-d_{n}+v_{n}),(4.8a)\\
d_{n+1}=prox_{\frac{1}{\nu_{n}}f_{1}}(Dx_{n+1}+v_{n}),~~~~~~~~~~~~~~~~~~~~~~~~~~~~~~~~~(4.8b)\\
v_{n+1}=v_{n}-(d_{n+1}-Dx_{n+1}),~~~~~~~~~~~~~~~~~~~~~~~~~~~~~~~~~~~(4.8c)
\end{array}
\right.
$$
where $\liminf_{n\rightarrow\infty}\delta_{n} > 0$ is a  dynamic
parameter.
\par
From (4.2a) and (4.2c), we can find out a relation between $y_{n}$
and $x_{n}$, given by
 $$x_{n} = y_{n}-\lambda_{n}D^{T} (v_{n}-v_{n+1}). $$
Then eliminating $x_{n}$, PDFP$^{2}O_{DS}$ can be expressed as
$$
\left\{
\begin{array}{l}
y_{n+1}=y_{n}-\lambda_{n}D^{T}(2v_{n}-v_{n-1})-\gamma_{n}\nabla f_{2}(y_{n}-\lambda_{n}D^{T}(v_{n}-v_{n-1})),(4.9a)\\
v_{n+1}=(I-prox_{\frac{\gamma_{n}}{\lambda_{n}}f_{1}})(Dy_{n+1}+v_{n}).~~~~~~~~~~~~~~~~~~~~~~~~~~~~~~~~~~~~~~(4.9b)
\end{array}
\right.
$$
By introducing the splitting variable $d_{n+1}$ in (4.9b), (4.9) can
be further expressed as

 Table 2 The comparisons among ASB, SIU and PDFP$^{2}O_{DS}$.

\begin{tabular}{llll}
\\
\hline\noalign{\smallskip}
\emph{} &\emph{ASB} &   & \\
\noalign{\smallskip}\hline\noalign{\smallskip} Form &
$x_{n+1}=(A^{T}A+\nu_{n}D^{T}D)^{-1}(A^{T}b+\nu_{n}D^{T}(d_{n}-v_{n}))$ &\\
     & $d_{n+1}=prox_{\frac{1}{\nu_{n}}f_{1}}(Dx_{n+1}+v_{n})$&\\
     & $v_{n+1}=v_{n}-(d_{n+1}-Dx_{n+1})$&\\
Convergence & $\liminf_{n\rightarrow\infty}\nu_{n}>0$ &  \\
\noalign{\smallskip}\hline
\emph{} &\emph{SIU} &   & \\
\noalign{\smallskip}\hline\noalign{\smallskip} Form &
$x_{n+1}=x_{n}-\delta_{n}A^{T}(Ax_{n}-b)-\delta_{n}\nu_{n}D^{T}(Dx_{n}-d_{n}+v_{n})$ &\\
     & $d_{n+1}=prox_{\frac{1}{\nu_{n}}f_{1}}(Dx_{n+1}+v_{n})$&\\
     & $v_{n+1}=v_{n}-(d_{n+1}-Dx_{n+1})$&\\
Convergence &
$\liminf_{n\rightarrow\infty}\nu_{n}>0$&  \\
 &      $0<\liminf_{n\rightarrow\infty}\delta_{n}\leq\limsup_{n\rightarrow\infty}\delta_{n}\leq1/\lambda_{\max}(A^{T}A+DD^{T})$ &  \\
 \noalign{\smallskip}\hline
\emph{} &\emph{PDFP$^{2}O_{DS}$} &   & \\
\noalign{\smallskip}\hline\noalign{\smallskip} Form
&$x_{n+1}=x_{n}-\delta_{n}A^{T}(Ax_{n}-b)-\delta_{n}\nu_{n}D^{T}(Dx_{n}-d_{n}+v_{n})$\\
& $-\delta_{n}^{2}\nu_{n}A^{T}AD^{T}(d_{n}-Dx_{n})$ \\
     & $d_{n+1}=prox_{\frac{1}{\nu_{n}}f_{1}}(Dx_{n+1}+v_{n})$&\\
     & $v_{n+1}=v_{n}-(d_{n+1}-Dx_{n+1})$&\\
Convergence &
$0<\liminf_{n\rightarrow\infty}\delta_{n}\leq\limsup_{n\rightarrow\infty}\delta_{n}<2/\lambda_{\max}(A^{T}A)$&  \\
 &      $0<\liminf_{n\rightarrow\infty}\delta_{n}\nu_{n}\leq\limsup_{n\rightarrow\infty}\delta_{n}\nu_{n}\leq1/\lambda_{\max}(DD^{T})$ &  \\
\noalign{\smallskip}\hline
\end{tabular}\\

$$
\left\{
\begin{array}{l}
y_{n+1}=y_{n}-\lambda_{n}D^{T}(Dy_{n}-d_{n}+v_{n})-\gamma_{n}\nabla f_{2}(y_{n}-\lambda_{n}D^{T}(Dy_{n}-d_{n})),(4.10a)\\
d_{n+1}=prox_{\frac{1}{\nu_{n}}f_{1}}(Dy_{n+1}+v_{n}),~~~~~~~~~~~~~~~~~~~~~~~~~~~~~~~~~~~~~~~~~~~~~~~~~~(4.10b)\\
v_{n+1}=v_{n}-(d_{n+1}-Dy_{n+1}).~~~~~~~~~~~~~~~~~~~~~~~~~~~~~~~~~~~~~~~~~~~~~~~~~~~~(4.10c)
\end{array}
\right.
$$
For $f_{2}(x) = \frac{1}{2}\|Ax-b\|^{2}_{2}$ , $\nabla f_{2}(x) =
A^{T} (Ax-b)$. By changing the order and letting
$\gamma_{n}=\delta_{n}$, $\lambda_{n}=\delta_{n}\nu_{n}(\forall
n\in\mathbb{N})$, (4.10) becomes
$$
\left\{
\begin{array}{l}
y_{n+1}=y_{n}-\delta_{n}A^{T}(Ay_{n}-b)-\delta_{n}\nu_{n}D^{T}(Dy_{n}-d_{n}+v_{n})\\
-\delta_{n}^{2}\nu_{n}A^{T}AD^{T}(d_{n}-By_{n}),~~~~~~~~~~~~~~~~~~~~~~~~~~~~~~~~~~~~~~~~~~~~~~~~~~~~~~(4.11a)\\
d_{n+1}=prox_{\frac{1}{\nu_{n}}f_{1}}(Dy_{n+1}+v_{n}),~~~~~~~~~~~~~~~~~~~~~~~~~~~~~~~~~~~~~~~~~~~~~~~~~(4.11b)\\
v_{n+1}=v_{n}-(d_{n+1}-Dy_{n+1}).~~~~~~~~~~~~~~~~~~~~~~~~~~~~~~~~~~~~~~~~~~~~~~~~~~~(4.11c)
\end{array}
\right.
$$
We can easily see that equation (4.7a) in ASB is approximated by
(4.10a). Although it seems that PDFP$^{2}O_{DS}$ requires more
computation in (4.10a) than SIU in (4.8a), PDFP$^{2}O_{DS}$ has the
same computation cost as that of SIU if the iterations are
implemented cleverly. For the reason of comparison, we can change
the variable $y_{n}$ to $x_{n}$ in (4.10). Table 2 gives the
summarized comparisons among ASB, SIU and PDFP$^{2}O_{DS}$. We note
that the only difference of SIU and PDFP$^{2}O_{DS}$ is in the first
step. As two algorithms converge, the algorithm  PDFP$^{2}O_{DS}$
behaves asymptotically the same as SIU since $d_{n}-Dx_{n}$
converges to 0. The parameters $\delta_{n}$ and $\nu_{n}$ satisfy
respectively different conditions to ensure the convergence.

\section{Numerical experiments}

In this section, we compare our proposed algorithm with the
state-of-the-art methods of PDFP$^{2}$O in the CT image
reconstruction problem. The test image is the standard benchmark
Shepp-Logan phantom (see Figure \ref{exper1}) with size of
$256\times 256$ and the pixels values vary from $0$ to $1$. All
experiments were performed under Windows 7 and MATLAB (R2009a)
running on a desktop with an Intel Core 2 Quad cpu and 2GB of RAM.

We use the toolbox of AIRTools to create 2D tomography test
problems. In the experiment setting, the projection angle is chosen
from $0$ to $175$ degrees in increments of $10$ degrees and the
number of parallel rays in each angle is $p=362$. We add Gaussian
white noise $e$ of relative magnitude $\|e\| / \|Ax_{true}\| =
0.01$.

\begin{figure}[H]
\setlength{\floatsep}{0pt} \setlength{\abovecaptionskip}{-20pt}
\centering
    \scalebox{0.6}{\includegraphics{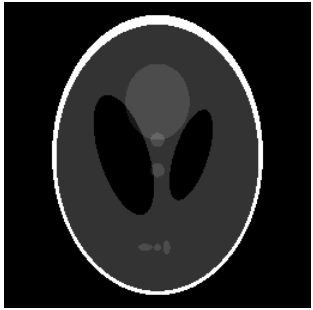}}
  \caption[]{The original Shepp-Logan phantom image}%
\label{exper1}
\end{figure}

The performances were evaluated in terms of the mean signal-to-noise
ratio (SNR) and the relative error (RelErr). The definitions of
 SNR and RelErr are given as follows:
 \begin{equation}
SNR = 20 log10 \left(  \frac{\|x_{true}\|}{\| x - x_{true} \|}
\right),
 \end{equation}
and
\begin{equation}
RelErr = \frac{\|x-x_{true}\|^2}{\| x_{true}  \|^2},
\end{equation}
 where $x$ and $x_{true}$ are the reconstructed image and
original image, respectively.

We follow the paper of [1] to choose the parameters for the
PDFP$^{2}$O. That is, the $\gamma = 2/\beta$, where $\beta$ is the
Lipschitz constant, and $\lambda = 1/8$. For our proposed algorithm,
we choose the dynamic stepsize $\gamma_n$ as follows:
\begin{equation}
\gamma_n = \frac{f_2(x_n)}{\|\nabla f_2(x_n)\|^2},
\end{equation}
where $f_2(x_n)= \|Ax_n -b\|^2$.

We tested anisotropic total variation and isotropic
 total variation regularization term and found the performance of anisotropic total variation
slightly better than isotropic
 total variation. Therefore, we only present results using anisotropic total variation here.
The reconstructed image is shown in Figure \ref{exper1}. As we can
see, both the algorithms achieve the good performance to reconstruct
the original image.

\begin{figure}[H]
\setlength{\floatsep}{0pt} \setlength{\abovecaptionskip}{-40pt}
\centering
    \scalebox{0.7}{\includegraphics{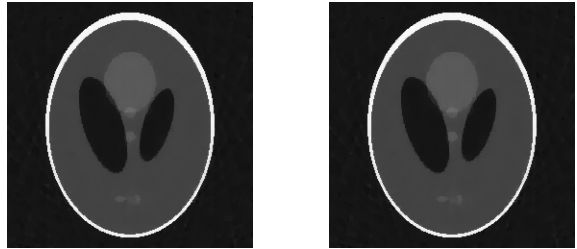}}
  \caption[]{The image reconstructed by the PDFP$^{2}$O and PDFP$^{2}$O$_{DS}$. Their SNR are 23.43 and 23.42 (db), respectively. }%
\label{exper1}
\end{figure}

\begin{figure}[H]
\setlength{\floatsep}{0pt} \setlength{\abovecaptionskip}{0pt}
\centering
    \scalebox{0.7}{\includegraphics{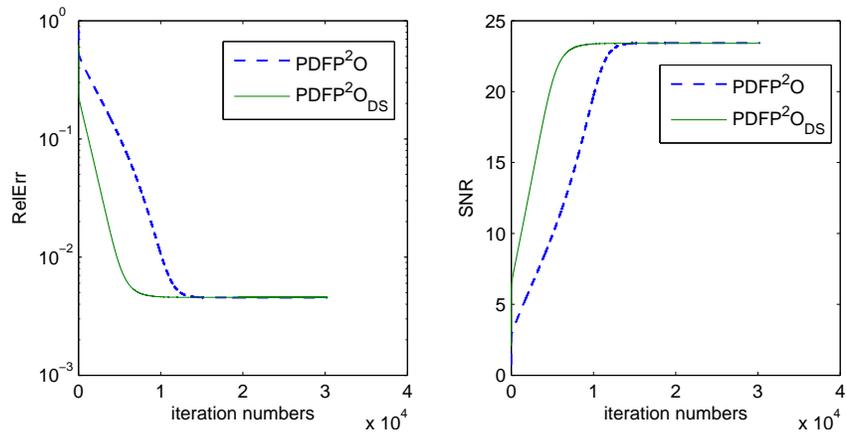}}
  \caption[]{The comparison of SNR and RelErr between PDFP$^{2}$O and PDFP$^{2}$$O_{DS}$}%
\label{expersnr1}
\end{figure}

We can see from Figure \ref{expersnr1} that the proposed algorithm
perform better than the PDFP$^{2}$O. Since the dynamic stepsize was
introduced in PDFP$^{2}$O$_{DS}$, it converges faster than the
original with constant stepsize. The more details of the choice of parameters $\gamma_{n}$ and $\lambda_{n}$ can be found in [25].

\noindent \textbf{Acknowledgements}

This work was supported by the National Natural Science Foundation
of China (11131006, 41390450, 91330204, 11401293), the National
Basic Research Program of China (2013CB 329404), the Natural Science
Foundations of Jiangxi Province (CA20110\\
7114, 20114BAB 201004).

\end{document}